\documentclass[12pt, a4paper]{amsart}

\usepackage[english]{babel}
\usepackage{amsmath}
\usepackage{amssymb}
\usepackage{bm}
\usepackage{amscd} 
\usepackage{microtype} 
\usepackage{verbatim} 
\usepackage{color}  
\usepackage{tikz-cd}\usetikzlibrary{babel} 
\usepackage{enumitem}
\usepackage{mathtools} 
\usepackage{cite}
\usepackage{hyperref} 
\usepackage[initials,msc-links,backrefs]{amsrefs} 
\usepackage[all]{xy}

\usepackage[OT2, T1]{fontenc}

\title[Bounded cohomology is not profinite]{Bounded cohomology is not\\ a profinite invariant}

 \author[D. Echtler]{Daniel Echtler}
 \author[H. Kammeyer]{Holger Kammeyer}
 
 \address{Mathematical Institute, University of D{\"u}sseldorf, Germany}
 \email{daniel.echtler@hhu.de}
 \email{holger.kammeyer@hhu.de}
 

\theoremstyle{plain}
\newtheorem{theorem}[equation]{Theorem}
\newtheorem{lemma}[equation]{Lemma}

\newtheorem{proposition}[equation]{Proposition}

\theoremstyle{definition}

\newtheorem*{definition*}{Definition}
\newtheorem{remark}[equation]{Remark}

\newtheorem*{observation*}{Observation}

\providecommand{\ignore}[1]{}

\providecommand{\R}{\mathbb{R}}
\providecommand{\Q}{\mathbb{Q}}
\providecommand{\Z}{\mathbb{Z}}

\providecommand{\C}{\mathbb{C}}

\DeclareMathOperator\Sp{Sp}
\DeclareMathOperator\Ad{Ad}
\DeclareMathOperator\Aut{Aut}

\newcommand*{\arXiv}[1]{ \href{http://www.arxiv.org/abs/#1}{arXiv:\textbf{#1}}}

\begin{document}

\begin{abstract}
  We construct pairs of residually finite groups with isomorphic profinite completions such that one has non-vanishing and the other has vanishing real second bounded cohomology.  The examples are lattices in different higher rank simple Lie groups.  Using Galois cohomology, we actually show that \(\operatorname{SO}^0(n,2)\) for \(n \ge 6\) and the exceptional groups \(E_{6(-14)}\) and \(E_{7(-25)}\) constitute the complete list of higher rank Lie groups admitting such examples.
\end{abstract}

\maketitle

\section{Introduction}

A group invariant is called \emph{profinite} if it agrees for any two finitely generated residually finite groups \(\Gamma, \Lambda\) with isomorphic profinite completions.  A standard example is the abelianization \(H_1(\Gamma)\), a more sophisticated example is largeness~\cite{Lackenby:large}.  It seems, however, that more often than not, group invariants fail to be profinite.  Kazhdan's property~\((T)\)~\cite{Aka:kazhdan}, higher \(\ell^2\)-Betti numbers~\cite{Kammeyer-Sauer:spinor-groups}, Euler characteristic and \(\ell^2\)-torsion~\cite{Kammeyer-et-al:profinite-invariants}, amenability~\cite{Kionke-Schesler:amenability}, finiteness properties~\cite{Lubotzky:finiteness-properties}, and most recently Serre's Property~FA~\cite{Cheetham-West-et-al:property-fa} are all known not to be profinite.  This list is by no means exhaustive and \emph{bounded cohomology} is another item:

\begin{lemma} \label{lemma:not-profinite}
Let \(\Gamma = \operatorname{Spin}(7,2)(\Z)\) and \(\Lambda = \operatorname{Spin}(3,6)(\Z)\).  Then \(\widehat{\Gamma} \cong \widehat{\Lambda}\) but \(H^2_b(\Gamma; \R) \cong \R\) while \(H^2_b(\Lambda; \R) \cong 0\).
\end{lemma}

For an ad hoc definition of the spinor group \(\operatorname{Spin}(q)(\mathcal{O})\)  of a quadratic form \(q\) over an integral domain \(\mathcal{O}\), we refer to \cite{Kammeyer-Sauer:spinor-groups}*{Section~3}.  The profinite completion \(\widehat{\Gamma}\) of a group \(\Gamma\) is the projective limit of the inverse system of finite quotient groups of \(\Gamma\).  The bounded cohomology \(H^*_b(\Gamma; \R)\) with real coefficients of a discrete group \(\Gamma\) is the cohomology of the cochain complex \(\ell^\infty(\Gamma^{*+1}, \R)^\Gamma\), of bounded functions \(\Gamma^{*+1} \rightarrow \R\) which are constant on the orbits of the diagonal \(\Gamma\)-action on \(\Gamma^{*+1}\), with the usual differential.  The reader may consult~\cite{Frigerio:bounded-cohomology} for further details.

Let us quickly prove the lemma.  If the Witt index of an integral quadratic form \(q\) is at least two, then \(\operatorname{Spin}(q)(\Z)\) has the \emph{congruence subgroup property} (CSP).  This implies that
\[ \widehat{\operatorname{Spin}(q)(\Z)} \cong \operatorname{Spin}(q)(\widehat{\Z}) \cong \textstyle \operatorname{Spin}(q)(\prod_p \Z_p) \cong \textstyle \prod_p \operatorname{Spin}(q)(\Z_p). \]
The standard forms of signature \((7,2)\) and \((3,6)\) are isometric over \(\Z_p\) for all (finite) primes \(p\) because standard quadratic form theory shows
\[ x_1^2 + x_2^2 + x_3^2 + x_4^2 \ \  \cong_{\Z_p} \ -x_1^2 -x_2^2 -x_3^2 -x_4^2. \]
Since \(\operatorname{Spin}(q)(\mathcal{O})\) is functorial in isometries \(q \cong_\mathcal{O} q'\) \cite{Kammeyer-Sauer:spinor-groups}*{Lemma~7}, we conclude \(\widehat{\Gamma} \cong \widehat{\Lambda}\).  But the second bounded cohomology of higher rank lattices was computed by Monod--Shalom~\cite{Monod-Shalom:cocycle-superrigidity}*{Theorem~1.4} and it turns out that \(H^2_b(\Gamma, \R) \cong \R\) whereas \(H^2_b(\Lambda, \R) \cong 0\).

\medskip
This counterexample to the profiniteness of bounded cohomology is similar in spirit to M.\,Aka's counterexamples to the profiniteness of Property~(T) in~\cite{Aka:kazhdan}.  Experts might have been aware of it but we could not find a reference.  In any case, it has prompted us to investigate thoroughly how often or rare such pairs of groups occur among higher rank lattices.  The purpose of this article is to give the complete picture.

\medskip
To state the result, let us agree that by a \emph{higher rank Lie group}, we mean the group of real points \(G = \mathbf{G}(\R)\) of a connected almost \(\R\)-simple linear algebraic \(\R\)-group \(\mathbf{G}\) with \(\operatorname{rank}_\R \mathbf{G} \ge 2\).  We say that \(G\) \emph{exhibits non-profinite second bounded cohomology} if there exists a lattice \(\Gamma \le G\) and another lattice \(\Lambda \le H\) in some other higher rank Lie group~\(H\) such that \(\widehat{\Gamma} \cong \widehat{\Lambda}\) and such that \(H^2_b(\Gamma; \R) \not\cong 0\) while \(H^2_b(\Lambda; \R) \cong 0\) .

\begin{theorem} \label{thm:main-theorem}
Let \(G\) be a higher rank Lie group.  Then \(G\) exhibits non-profinite second bounded cohomology if and only if it is isogenous to
    \[ \operatorname{SO}^0(n,2) \text{ for } n \ge 6, \text{ or to } \ E_{6(-14)}, \text{ or to } E_{7(-25)}. \]
  \end{theorem}
  
  Here we call two simple Lie groups \emph{isogenous} if they have isomorphic Lie algebras.  Note that a finite index subgroup of the group \(\Gamma\) from Lemma~\ref{lemma:not-profinite} is a lattice in \(\operatorname{SO}^0(7,2)\).  The lemma actually provides the easiest possible example because the examples of lattices that we construct in \(\operatorname{SO}^0(6,2)\) come from triality forms of type \(D_4\), as we will see.

  \medskip
  Let us outline the proof of Theorem~\ref{thm:main-theorem}.  The key result is the aforementioned theorem due to Monod--Shalom \cite{Monod-Shalom:cocycle-superrigidity}*{Theorem~1.4} which extends a previous result of Burger--Monod \cite{Burger-Monod:bounded-cohomology}*{Corollary~1.6}: for a lattice \(\Gamma \le G\) in a higher rank Lie group, we have \(H^2_b(\Gamma; \R) \cong \R\) if \(\pi_1 G\) is infinite and \(H^2_b(\Gamma; \R) \cong 0\) otherwise.  It is well-known that \(\pi_1 G\) is infinite if and only if the symmetric space \(G/K\) associated with \(G\) is \emph{hermitian}~\cite{Helgason:differential-geometry}*{Theorem VIII.6.1, p.\,381}.  The classification of hermitian symmetric spaces is long-established~\cite{Helgason:differential-geometry}*{Section~X.6.3, p.\,518}.  The irreducible hermitian symmetric spaces of higher rank are precisely the symmetric spaces of the simple Lie groups
  \begin{gather*}
    \operatorname{SU}(n,m) \text{ for } n,m \ge 2, \ \operatorname{SO}^0(n,2) \text{ for } n \ge 3, \ \operatorname{SO}^*(2n) \text{ for } n \ge 4, \\
    \operatorname{Sp}(n, \R) \text{ for } n \ge 2, \ E_{6(-14)}, \text{ and } \ E_{7(-25)}.
    \end{gather*}
  Here \(\operatorname{SO}^0(n,2)\) is the identity component of the determinant one matrices that preserve the standard quadratic form of signature \((n,2)\).  The group \(\operatorname{SO}^*(2n)\) is the quaternionic special orthogonal group as defined for instance in~\cite{Witte-Morris:arithmetic}*{Example~A2.4.2, p.\,430}.  Using the symbol ``\(\approx\)'' for isogenous groups, we have the accidental isogenies \(\operatorname{SO}^0(3,2) \approx \operatorname{Sp}(2;\R)\), \(\operatorname{SO}^0(4,2) \approx \operatorname{SU}(2,2)\), and \(\operatorname{SO}^0(6,2) \approx \operatorname{SO}^*(8)\).  Let \(G\) be one of the groups in the list and let \(\Gamma \le G\) be any lattice.  By Margulis arithmeticity, we may assume \(\Gamma\) is an arithmetic subgroup of a simply-connected simple algebraic group \(\mathbf{G}\) over some totally real number field \(k\) such that \(G \approx \mathbf{G}(k_v)\) for some real place \(v\) of \(k\) and such that \(\mathbf{G}\) is anisotropic at all other infinite places of \(k\).  The congruence subgroup property translates the question whether there exists \(\Lambda \le H\) as in the theorem to whether there exists a simply-connected simple \(l\)-group \(\mathbf{H}\), an isomorphism \(\mathbb{A}^f_k \cong \mathbb{A}^f_l\) of topological rings between the finite adele rings of \(k\) and \(l\), and a corresponding isomorphism \(\mathbf{G}(\mathbb{A}^f_k) \cong \mathbf{H}(\mathbb{A}^f_l)\) such that \(\mathbf{H}\) is anisotropic at all but one infinite place where it should be isogenous to a higher rank Lie group outside the list.  The technical achievement of this paper, beside filling in the details of the arguments thus far, is to solve this problem by Galois cohomological methods.

  \medskip
We conclude the introduction with some comments on related work and open questions.  By definition, bounded cohomology comes with a \emph{comparison map}
\(H^*_b(\Gamma; \R) \longrightarrow H^*(\Gamma; \R)\).  The kernel of this
homomorphism is denoted by \(EH^*_b(\Gamma; \R)\) and is called \emph{exact
bounded cohomology}.  In degree two, we have the well-known interpretation that
\(EH^2_b(\Gamma; \R)\) detects \emph{non-trivial quasimorphisms}: maps
\(f \colon \Gamma \rightarrow \R\) for which there exists \(D > 0\) with
\[ |f(gh) -f(g) -f(h)| \le D \] for all \(g,h \in \Gamma\) such that \(f\) is
not at bounded distance from an honest
homomorphism~\cite{Frigerio:bounded-cohomology}*{Section~2.3}.  For lattices in
higher rank linear Lie groups it was verified by M.\,Burger and
N.\,Monod~\cite{Burger-Monod:continuous}*{Theorem~21} that the comparison map
in degree two is injective. So the question whether the existence of
non-trivial quasimorphisms is a profinite property remains open for now.

However, for lattices in rank one groups, the situation is different. On the one hand, a result of Fujiwara shows that such lattices
admit many quasimorphisms \cite{Fujiwara:hyperbolic}. On the other hand, Serre
conjectured that these lattices should not have the congruence subgroup
property, so an important ingredient to construct groups with isomorphic
profinite completions would be missing. Yet, lattices in the rank one groups $\Sp(n, 1)$ and $F_{4(-20)}$ share many properties with higher rank lattices which might suggest that they in fact do have CSP~\cite{Lubotzky:more-rigid-groups}*{Section~4}.

If \(F_{4(-20)}\) has CSP, then having non-trivial exact second bounded cohomology, or equivalently having non-trivial quasimorphisms, is \emph{not} a profinite property.  Indeed, let $\mathbf{F_4}$ be the unique simply-connected absolutely almost simple $\Q$-split linear algebraic \(\Q\)-group of type $F_4$.  Since the Dynkin diagram of type $F_4$ has no symmetries and the center of $F_4$ is trivial, we have \(\mathbf{F_4} \cong \Ad \mathbf{F_4} \cong \Aut \mathbf{F_4}\).  Therefore, the Hasse principle for simply-connected groups gives \(H^1(\Q, \Aut \mathbf{F_4}) \cong H^1(\R, \Aut \mathbf{F_4})\), meaning every real form of type \(F_4\) comes with a unique \(\Q\)-structure.  Moreover, any two \(\Q\)-groups of type \(F_4\) are \(\Q_p\)-split and hence \(\Q_p\)-isomorphic for all finite primes \(p\) by Kneser's theorem~\cite{Kneser:galois}.  Therefore, if \(F_{4(-20)}\) has CSP, we can find arithmetic lattices \(\Gamma \le F_{4(-20)}\) and \(\Lambda \le F_{4(4)}\) with \(\widehat{\Gamma} \cong \widehat{\Lambda}\) but \(EH_b^2(\Gamma; \R) \neq 0\) while \(EH_b^2(\Lambda; \R) = 0\). 

\medskip
Another notion from this circle of ideas is \emph{Ulam stability}.  Here, instead of \(\R\), we consider unitary groups \(U(n) = \{ A \in \operatorname{GL}_n(\C) \colon A^* = A^{-1} \}\) and define an \emph{\(\varepsilon\)-homomorphism} as a map \(f \colon \Gamma \rightarrow U(n)\) such that \(\|f(gh) - f(g)f(h)\| \le \varepsilon\) holds for all \(g,h \in \Gamma\) where \(\|\cdot\|\) denotes the operator norm in \(M_n(\C)\).  We say that \(\Gamma\) is \emph{uniformly \(U(n)\)-stable} if there exists a function \(\delta = \delta(\varepsilon)\) with \(\lim_{\varepsilon \rightarrow 0} \delta(\varepsilon) = 0\) such that for all \(\varepsilon\)-homomorphisms \(f \colon \Gamma \rightarrow U(n)\), there exists an honest homomorphism \(F \colon \Gamma \rightarrow U(n)\) such that for all \(g \in \Gamma\), we have \(\|f(g)-F(g)\| \le \delta(\varepsilon)\).  We say that \(\Gamma\) is \emph{Ulam stable} if it is \(U(n)\)-stable for all \(n \ge 1\).

It is a straightforward consequence of the Burger--Monod theorem that lattices in higher rank linear Lie groups are uniformly \(U(1)\)-stable \cite{Glebsky-et-al:asymptotic-cohomology}*{Theorem~1.0.11}.  Most recently, L.\,Glebsky, A.\,Lubotzky, N.\,Monod, and B.\,Rangarajan proved the far reaching generalization that lattices in many higher rank semisimple Lie groups \(G\) are Ulam stable~\cite{Glebsky-et-al:asymptotic-cohomology}*{Theorem~0.0.5}.  In fact, they are even uniformly stable with respect to more general metrics on \(U(n)\).  Interestingly, the technical condition ``property-\(G(\mathcal{Q}_1, \mathcal{Q}_2)\)'' that \(G\) needs to satisfy to conclude this stability of lattices fails for~\(\operatorname{SO}^0(n,2)\), \(E_{6(-14)}\) and~\(E_{7(-25)}\) by \cite{Glebsky-et-al:asymptotic-cohomology}*{Theorem~0.0.6} but it holds true for \(E_{6(2)}\) and \(E_{7(7)}\)~\cite{Glebsky-et-al:asymptotic-cohomology}*{Proof of Proposition~6.3.6}.  (It seems to be open whether it holds for groups of type \(\operatorname{SO}^0(p,q)\) with \(p,q \ge 3\).)  We will see below that \(E_{6(2)}\) and \(E_{7(7)}\) contain lattices which are profinitely isomorphic to lattices in \(E_{6(-14)}\) and \(E_{7(-25)}\), respectively.  The authors actually entertain the idea that property-\(G(\mathcal{Q}_1, \mathcal{Q}_2)\) might be necessary for Ulam stability~\cite{Glebsky-et-al:asymptotic-cohomology}*{Section~7}.  If that was true, our theorem would thus have the corollary that Ulam stability is not a profinite property.

\medskip
  In Section~\ref{section:if}, we prove the ``if part'' of Theorem~\ref{thm:main-theorem} and in Section~\ref{section:onlyif}, we prove the ``only if'' part.  The authors acknowledge financial support from the German Research Foundation via the Research Training Group ``Algebro-Geometric Methods in Algebra, Arithmetic, and Topology'', DFG 284078965.  The second author is additionally grateful for financial support within the Priority Program ``Geometry at Infinity'', DFG 441848266.  We are indebted to Francesco Fournier-Facio, Ryan Spitler, and the referee for helpful discussions and suggestions.

\section{Proof of Theorem~\ref{thm:main-theorem}---``if part''} \label{section:if}

In this section, we show that higher rank Lie groups isogenous to \(\operatorname{SO}^0(n,2)\) for \(n \ge 6\), \(E_{6(-14)}\), or \(E_{7(-25)}\) exhibit non-profinite second bounded cohomology.  We start with the group \(E_{7(-25)}\).

\medskip
Fix a prime number \(p_0\).  Let \(\mathbf{E_7}\) be the unique simply-connected absolutely almost simple \(\Q\)-split linear algebraic \(\Q\)-group of type \(E_7\).  The center \(Z(\mathbf{E_7}) \cong \mu_2\) is isomorphic to the algebraic group of ``second roots of unity'' \cite{Platonov-Rapinchuk:algebraic-groups}*{Table on p.\,332}, so that \(\mu_2(K) = \{\pm 1\}\) for any field extension \(K/\Q\).  By \cite{Serre:galois-cohomology}*{Section~I.5.7}, the corresponding equivariant short exact sequence of \(\operatorname{Gal}(\Q)\)-groups
\[ 1 \longrightarrow \mu_2 \longrightarrow \mathbf{E_7} \longrightarrow \operatorname{Ad} \mathbf{E_7} \longrightarrow 1 \]
and functoriality yield a commuting diagram of Galois cohomology sets
\[
  \begin{tikzcd}
    {\bigoplus\limits_{p \neq p_0} H^1(\Q_p, \mathbf{E_7})} \arrow[r, "\bigoplus \pi_p"] & {\bigoplus\limits_{p \neq p_0} H^1(\Q_p, \operatorname{Ad} \mathbf{E_7})} \arrow[r, "\bigoplus \Delta_p"] & {\bigoplus\limits_{p \neq p_0} H^2(\Q_p, \mu_2)} \\
    {H^1(\Q, \mathbf{E_7})} \arrow[r, "\pi"] \arrow[u] \arrow[d]               & {H^1(\Q,\operatorname{Ad} \mathbf{E_7})} \arrow[r, "\Delta"] \arrow[u, "f"] \arrow[d, swap, "f_{p_0}"]                & {H^2(\Q, \mu_2)} \arrow[u, "b"] \arrow[d, swap, "b_{p_0}"]      \\
    {H^1(\Q_{p_0}, \mathbf{E_7})} \arrow[r, "\pi_{p_0}"]                       & {H^1(\Q_{p_0},\operatorname{Ad} \mathbf{E_7})} \arrow[r, "\Delta_{p_0}"]                        & {H^2(\Q_{p_0}, \mu_2)}                   
    \end{tikzcd}
  \]
which is exact at the middle term of each row.  The direct sums in the upper row denote the subsets of the cartesian products consisting of elements with all but finitely many coordinates equal to the unit class.  We agree that the infinite prime \(\infty\) with \(\Q_\infty = \R\) is included.  Let us collect some information on this diagram

\begin{proposition} \label{prop:injective-surjective}
  The map \(f\) is surjective and \(\Delta_{p_0}\) has trivial kernel.
\end{proposition}

\begin{proof}
  The surjectivity of \(f\) is implicit in the work of Borel--Harder~\cite{Borel-Harder:existence} but can be cited explicitly from Prasad--Rapinchuk \cite{Prasad-Rapinchuk:isotropic}*{Proposition~1}.  The map \(\Delta_{p_0}\) has trivial kernel by exactness of the lower sequence and because \(H^1(\Q_{p_0}, \mathbf{E_7}) = 0\) by a result of M.\,Kneser~\cite{Kneser:galois}.
\end{proof}

\begin{proposition} \label{prop:image}
 The image of \(\pi_\R \colon H^1(\R, \mathbf{E_7}) \longrightarrow H^1(\R, \operatorname{Ad} \mathbf{E_7})\) consists of two elements, corresponding to the real forms \(E_{7(7)}\) and \(E_{7(-25)}\).
\end{proposition}

\begin{proof}
    This was already observed as part of the investigation in~\cite{Kammeyer-Spitler:chevalley} but let us give a direct argument for the convenience of the reader.  By \cite{Serre:galois-cohomology}*{Section~I.5.7}, the map \(\pi_\R\) sits in the exact sequence
    \[ \mathbf{E_7}(\R) \xrightarrow{\ p \ } \operatorname{Ad} \mathbf{E_7}(\R) \xrightarrow{\ \delta\ } H^1(\R, \mu_2) \longrightarrow H^1(\R, \mathbf{E_7}) \xrightarrow{\pi_{\R}} H^1(\R, \operatorname{Ad} \mathbf{E_7}). \]
    By \cite{Serre:galois-cohomology}*{Corollary~2, Section~I.5.6}, the map \(\delta\) is a group homomorphism, hence \(\operatorname{im} \delta \cong \operatorname{Ad} \mathbf{E_7}(\R) / p(\mathbf{E_7}(\R))\) by exactness.  But the \(\R\)-points of a simply-connected semisimple \(\R\)-group form a connected Lie group~\cite{Margulis:discrete-subgroups}*{Remark~(2), p.\,52}, whereas \(\operatorname{Ad} \mathbf{E_7}(\R)\) has two connected components according to~\cite{Adams-Taibi:galois-cohomology}*{Table~5, p.\,1095}.  Since \(H^1(\R, \mu_2) \cong \Z / 2\), it follows that \(\delta\) is surjective so \(\operatorname{ker} \pi_\R = 1\) by exactness.  By~\cite{Adams-Taibi:galois-cohomology}*{Table~3, p.\,1094}, the set \(H^1(\R, \mathbf{E_7})\) has two elements, so \(\pi_\R\) is injective.  If \([a] \in H^1(\R, \mathbf{E_7})\) denotes the non-trivial class, we conclude from \cite{Serre:galois-cohomology}*{Corollary~2, Section~I.5.5} that the \({}_a \operatorname{Ad} \mathbf{E_7}(\R)\)-action on \(H^1(\R, \mu_2) \cong \Z/2\) is transitive, hence \({}_a \operatorname{Ad} \mathbf{E_7}(\R)\) must be disconnected.  By \cite{Adams-Taibi:galois-cohomology}*{Table~5, p.\,1095}, the hermitian form \(E_{7(-25)}\) is the only non-split disconnected form, so \({}_a \operatorname{Ad} \mathbf{E_7}(\R)\) is of type \(E_{7(-25)}\).  This shows \(\pi_\R([a]) \in H^1(\R, \operatorname{Ad} \mathbf{E_7})\) corresponds to the real form \(E_{7(-25)}\).  Of course, the unit class \(1 \in H^1(\R, \operatorname{Ad} \mathbf{E_7})\) corresponds to the split form \(E_{7(7)}\).
\end{proof}

  The surjectivity of \(f\) according to Proposition~\ref{prop:injective-surjective} lets us find classes \(\alpha, \beta \in H^1(\Q, \operatorname{Ad} \mathbf{E_7})\) such that both \(\alpha\) and \(\beta\) split at all finite primes \(p\) except possibly at \(p_0\) and such that \(\alpha\) corresponds to the real form \(E_{7(-25)}\) at \(\infty\), whereas \(\beta\) corresponds to the real form \(E_{7(7)}\) at \(\infty\).

  \begin{proposition} \label{prop:same-localization}
    Both \(\alpha\) and \(\beta\) also split at \(p_0\). 
  \end{proposition}

  \begin{proof}
    Proposition~\ref{prop:image} and exactness of the upper sequence in the above diagram show that \(\bigoplus \Delta_p(f(\alpha)) = \bigoplus \Delta_p(f(\beta)) = 1\).  By commutativity, we have \(b(\Delta(\alpha)) = b(\Delta(\beta)) = 1\).  The map \(b\) is injective as a special case of~\cite{Prasad-Rapinchuk:isotropic}*{Theorem~3.(2)}.  In fact, the injectivity is an immediate consequence of the extended Albert--Brauer--Hasse--Noether theorem from global class field theory, stating that we have a short exact sequence
    \begin{equation} \label{eq:brauer} 1 \longrightarrow \operatorname{Br}(\Q) \longrightarrow \bigoplus_p \operatorname{Br}(\Q_p) \xrightarrow{\ s \ } \Q/\Z \longrightarrow 1 \end{equation}
    where \(s\) sums up local invariants.  Indeed, \(H^2(\Q, \mu_2)\) is the subgroup \(\operatorname{Br}_2(\Q)\) of the Brauer group \(\operatorname{Br}(\Q) = H^2(\Q, \mathbf{GL_1})\) consisting of \mbox{order} two elements.  It follows that \(\Delta(\alpha) = \Delta(\beta) = 1\) so \(b_{p_0}(\Delta(\alpha)) = b_{p_0}(\Delta(\beta)) = 1\).  Since \(\Delta_{p_0}\) has trivial kernel by Proposition~\ref{prop:injective-surjective}, commutativity of the right lower square gives \(f_{p_0}(\alpha) = f_{p_0}(\beta) = 1\).\end{proof}

  As the Dynkin diagram of type \(E_7\) comes with no symmetries, the set \(H^1(\Q, \operatorname{Ad} \mathbf{E_7})\) classifies all \(\Q\)-forms of type \(E_7\).  The upshot of Proposition~\ref{prop:same-localization} is that the simply-connected \(\Q\)-forms \(\mathbf{G_1}\) and \(\mathbf{G_2}\) defined by \(\alpha\) and \(\beta\), respectively, are split and thus isomorphic over \(\Q_p\) for all finite primes \(p\).  It then follows from \cite{Kammeyer-Kionke:adelic}*{Lemma~2.5 and Lemma~2.6} that we also have an isomorphism \(\mathbf{G_1} \cong \mathbf{G_2}\) of group schemes over the finite adele ring \(\mathbb{A}^f_\Q\).  Moreover, neither \(\mathbf{G_1}\) nor \(\mathbf{G_2}\) is topologically simply-connected at the infinite place.  Indeed, we have \(Z(\mathbf{G_1}(\R)) \cong Z(\mathbf{G_2}(\R)) \cong \{ \pm 1 \}\).  But according to~\cite{Onishchik-Vinberg:lie-groups}*{Table~10, p.\,321}, the simply-connected real Lie group of type \(E_{7(7)}\) has cyclic center of order four while the simply-connected real Lie group of type \(E_{7(-25)}\) has infinite cyclic center. (The latter is actually a consequence of \(E_{7(-25)}\) giving rise to a hermitian symmetric space.)  By~\cite{Prasad-Rapinchuk:metaplectic-kernel}*{Main Theorem}, this implies that the metaplectic kernels of \(\mathbf{G_1}\) and \(\mathbf{G_2}\) are trivial.  The surjectivity result for the map \(f\) comes with an additional statement on the existence of isotropic preimages~\cite{Prasad-Rapinchuk:isotropic}*{Theorem~1\,(iii)} which allows us to assume that \(\operatorname{rank}_\Q \mathbf{G_1} = 3\) and \(\operatorname{rank}_\Q \mathbf{G_2} = 7\).  Therefore, the centrality of the congruence kernels of \(\mathbf{G}_1\) and \(\mathbf{G_2}\) follows from~\cite{Raghunathan:csp}.  Together with \cite{Platonov-Rapinchuk:algebraic-groups}*{Theorem~9.1 and Theorem~9.15}, we conclude that the congruence kernels of \(\mathbf{G_1}\) and \(\mathbf{G_2}\) are in fact trivial.

  \medskip
  Let \(\Gamma_0\) and \(\Lambda_0\) be arithmetic subgroups of \(\mathbf{G_1}\) and \(\mathbf{G_2}\), respectively, which we may assume intersect the center trivially.  Since the congruence kernel of \(\mathbf{G_1}\) is trivial, the profinite completion \(\widehat{\Gamma_0}\) agrees with the congruence completion \(\overline{\Gamma_0}\).  The latter is an open subgroup of \(\mathbf{G_1}(\mathbb{A}^f_\Q)\) by strong approximation.  Similarly, under the isomorphism \(\mathbf{G_1}(\mathbb{A}^f_\Q) \cong \mathbf{G_2}(\mathbb{A}^f_\Q)\), the group \(\widehat{\Lambda_0} = \overline{\Lambda_0}\) is embedded as another open subgroup in \(\mathbf{G_1}(\mathbb{A}^f_\Q)\).  We denote the open intersection of these two open subgroups by \(U \le \mathbf{G_1}(\mathbb{A}^f_\Q)\).  Then by \cite{Ribes-Zalesskii:profinite-groups}*{Proposition 3.2.2, p.\,80, and Lemma~3.1.4, p.\,77}, the groups \(\Gamma = \Gamma_0 \cap U\) and \(\Lambda = \Lambda_0 \cap U\) have finite index in \(\Gamma_0\) and \(\Lambda_0\), respectively, and \(\widehat{\Gamma} \cong U \cong \widehat{\Lambda}\).

  \medskip
  By the Borel--Harish-Chandra Theorem~\cite{Margulis:discrete-subgroups}*{Theorem~I.3.2.7, p.\,63}, \(\Gamma\) is a lattice in the Lie group \(\mathbf{G_1}(\R)\) while \(\Lambda\) is a lattice in \(\mathbf{G_2}(\R)\).  Since \(\Gamma\) does not meet the center of \(\mathbf{G_1}\), \(\Gamma\) is also a lattice in any central quotient of \(\mathbf{G_1}(\R)\).  As we explained in the introduction, the Burger--Monod--Shalom theorem gives \(H^2_b(\Gamma; \R) \cong \R\) and \(H^2_b(\Lambda; \R) \cong 0\) because \(\mathbf{G_1}(\R)\) has type \(E_{7(-25)}\) whereas \(\mathbf{G_2}(\R)\) has type \(E_{7(7)}\).  This completes the proof that Lie groups of type \(E_{7(-25)}\) exhibit non-profinite second bounded cohomology.

  \begin{remark}
    We are grateful to the anonymous referee who suggested to us the following alternative way to construct a \(\Q\)-form \(\mathbf{G_1}\) of type \(E_7\) which has type \(E_{7(-25)}\) at the real place and splits at all finite places.  Consider the standard quadratic form \(q = x_1 ^2 + \cdots + x_8^2\) of rank eight.  It corresponds to an element \(\alpha_q \in H^1(\Q, \mathbf{SO}(4,4))\) because the discriminant of \(q\) and the standard form of signature \((4,4)\) are both trivial.  The boundary map \(\delta^2 \colon H^1(\Q, \mathbf{SO}(4,4)) \longrightarrow \operatorname{Br}_2(\Q)\) associated with the short exact sequence
    \[ 1 \longrightarrow \mu_2 \longrightarrow \mathbf{Spin}(4,4) \longrightarrow \mathbf{SO}(4,4) \longrightarrow 1 \]
    is given by the Hasse--Witt invariant~\cite{Serre:galois-cohomology}*{III.3.2.b), p.\,141} so that \(\delta^2(\alpha_q) = w_2(q) = 1\).  Thus \(\alpha_q\) has a preimage \(\beta_q \in H^1(\Q, \mathbf{Spin}(4,4))\).  Via the obvious Dynkin diagram inclusion \(D_4 \subset E_7\), the class \(\beta_q\) maps to a class in \(H^1(\Q, \mathbf{E_7})\).  The image in \(H^1(\Q, \operatorname{Ad} \mathbf{E_7})\) thus defines a \(\Q\)-form \(\mathbf{G_1}\) of type \(E_7\) which splits at every finite place and has Satake--Tits index
    
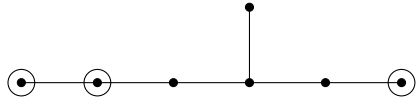
\begin{center}
      \begin{tikzpicture}
        \draw (0,0) -- (5, 0);
        \foreach \x in {0, ..., 5}{
          \draw[fill] (\x,0) circle (1.5pt);
        }
        \draw[fill] (3,0) -- (3,1) circle (1.5pt);

        \foreach \x in {0, 1, 5}{
          \draw (\x,0) circle (5pt);
        }
      \end{tikzpicture}
    \end{center}
    at the real place by construction.  This shows that \(\mathbf{G_1}(\R)\) is the Lie group \(E_{7(-25)}\).
  \end{remark}
  \medskip

  We now turn our attention to the group \(E_{6(-14)}\).  Again, we fix a rational prime \(p_0\).  Let \(\mathbf{{}^2 E_6}\) be any simply-connected absolutely almost simple quasisplit \(\Q\)-group which neither splits at \(p_0\) nor at \(\infty\).  Such a form exists according to~\cite{Borel-Harder:existence}*{Proposition p.\,58}.  The Hasse principle for adjoint groups~\cite{Platonov-Rapinchuk:algebraic-groups}*{Theorem~6.22, p.\,336} shows that the diagonal map
  \[ f \colon H^1(\Q, \operatorname{Ad} \mathbf{{}^2 E_6}) \longrightarrow \bigoplus_{p \le \infty} H^1(\Q_p, \operatorname{Ad} \mathbf{{}^2 E_6}) \]
  is injective.  But \(f\) is also surjective by~\cite{Prasad-Rapinchuk:isotropic}*{Proposition~1} because at~\(p_0\), the group \(\mathbf{{}^2 E_6}\) has no inner twist, meaning the set \(H^1(\Q_{p_0}, \operatorname{Ad} \mathbf{{}^2 E_6})\) is trivial~\cite{Platonov-Rapinchuk:algebraic-groups}*{Proposition~6.15.(1), p.\,334}.

  The set \(H^1(\R, \operatorname{Ad} \mathbf{{}^2 E_6})\) has three elements corresponding to the quasisplit form \(E_{6(2)}\), the hermitian form \(E_{6(-14)}\), and the compact form \(E_{6(-78)}\) as we infer one more time from~\cite{Adams-Taibi:galois-cohomology}*{Table~3, p.\,1094 and remark in 10.3}.  Let \(\alpha, \beta \in H^1(\Q, \operatorname{Ad} \mathbf{{}^2 E_6})\) be the unique classes corresponding to \(E_{6(-14)}\) and \(E_{6(2)}\) at the infinite place, respectively, and which are trivial at all finite places.  We denote the simply-connected \(\Q\)-forms of \(\mathbf{{}^2 E_6}\) corresponding to \(\alpha\) and \(\beta\) by \(\mathbf{G_1}\) and \(\mathbf{G_2}\).  By construction, \(\mathbf{G_1}\) and \(\mathbf{G_2}\) are isomorphic over \(\Q_p\) for all finite primes \(p\).  From the tables~\cite{Platonov-Rapinchuk:algebraic-groups}*{p.\,332} and~\cite{Onishchik-Vinberg:lie-groups}*{Table~10, p.\,321}, we infer that the centers of \(\mathbf{G_1}(\R)\) and \(\mathbf{G_2}(\R)\) have order three whereas the centers of the simply-connected real Lie groups of type \(E_{6(-14)}\) and \(E_{6(2)}\) are infinite cyclic and order six, respectively.  So again, the metaplectic kernels of \(\mathbf{G_1}\) and \(\mathbf{G_2}\) are trivial.  It follows anew from~\cite{Prasad-Rapinchuk:isotropic}*{Theorem~1\,(iii)} that \(\operatorname{rank}_\Q \mathbf{G_1} = 2\) and \(\operatorname{rank}_\Q \mathbf{G_2} = 4\) so that both congruence kernels are central by~\cite{Raghunathan:csp}.  Finally the case of type \({}^2 E_6\), which was still excluded in~\cite{Platonov-Rapinchuk:algebraic-groups}*{Theorem~9.1, p.\,512}, was meanwhile settled by P.\,Gille~\cite{Gille:kneser-tits}.  So \cite{Platonov-Rapinchuk:algebraic-groups}*{Theorem~9.15} shows that both \(\mathbf{G}_1\) and \(\mathbf{G}_2\) have trivial congruence kernel.  With these remarks, the rest of the argument goes through as before and we conclude that \(E_{6(-14)}\) exhibits non-profinite second bounded cohomology.

  \bigskip
  For the group \(\operatorname{SO}^0(6,2)\), we can actually argue similarly.  We let \(\mathbf{{}^3 D_4}\) be a simply-connected absolutely almost simple quasisplit \(\Q\)-group that localizes to the quasisplit triality form of type \({}^3 D_4\) at a fixed prime \(p_0\) and that splits at \(\infty\).  Note that the \(\Q\)-group \(\mathbf{{}^3 D_4}\) can either have outer type \({}^3 D_4\) or \({}^6 D_4\).  Since again \(H^1(\Q_{p_0}, \operatorname{Ad} \mathbf{{}^3 D_4})\) is trivial according to~\cite{Platonov-Rapinchuk:algebraic-groups}*{Proposition~6.15.(1), p.\,334}, we have a pointed bijection
  \[ f \colon H^1(\Q, \operatorname{Ad} \mathbf{{}^3 D_4}) \longrightarrow \bigoplus_{p \le \infty} H^1(\Q_p, \operatorname{Ad} \mathbf{{}^3 D_4}). \]
  Hence there exist \(\Q\)-forms \(\mathbf{G_1}\) and \(\mathbf{G_2}\) of \(\mathbf{{}^3 D_4}\) that localize to the inner forms \(\operatorname{SO}^0(6,2)\) and \(\operatorname{SO}^0(4,4)\) at the real place, respectively, and to the trivial inner twist of \(\mathbf{{}^3 D_4}\) at all finite places so that they are isomorphic over \(\Q_p\) for all \(p\).  We have \(\operatorname{rank}_\Q \mathbf{G_1} = 1\) and \(\operatorname{rank}_\Q \mathbf{G_2} = 2\), so both groups have central congruence kernel, this time by~\cite{Raghunathan:csp2}.  The tables~\cite{Platonov-Rapinchuk:algebraic-groups}*{p.\,332} and~\cite{Onishchik-Vinberg:lie-groups}*{Table~10, p.\,320} show that \(\mathbf{G_1}(\R)\) and \(\mathbf{G_2}(\R)\) have trivial center whereas the corresponding topological universal coverings have center isomorphic to \(\Z \oplus \Z / 2\) and \((\Z/2)^3\), respectively.  This implies that the metaplectic kernels of \(\mathbf{G_1}\) and \(\mathbf{G_2}\) are trivial and so are the congruence kernels.  Hence we can once more construct profinitely isomorphic lattices in \(\mathbf{G_1}(\R)\) and \(\mathbf{G_2}(\R)\) as we did in type \(E_7\) and we conclude that \(\operatorname{SO}^0(6,2) \approx \operatorname{SO}^*(8)\) exhibits non-profinite second bounded cohomology.

\bigskip
Finally, we observe that Lemma~\ref{lemma:not-profinite} generalizes effortlessly to the groups \(\Gamma = \operatorname{Spin}(n,2)(\Z)\) and \(\Lambda = \operatorname{Spin}(n-4,6)(\Z)\) for \(n \ge 7\).  Let \(\Gamma_0 \le \Gamma\) be a finite index subgroup which intersects the center of \(\operatorname{Spin}(n,2)(\R)\) trivially.  Setting \(\Lambda_0 = \Lambda \cap \overline{\Gamma_0} \subset \widehat{\Gamma} \cong \widehat{\Lambda}\), we have \(\widehat{\Gamma_0} \cong \widehat{\Lambda_0}\).  The group \(\Gamma_0\) is a lattice in every quotient group of \(\operatorname{Spin}(n,2)(\R)\) by a central subgroup so that all Lie groups isogenous to \(\operatorname{SO}^0(n,2)\) exhibit non-profinite second bounded cohomology for \(n \ge 7\).

  \section{Proof of Theorem~\ref{thm:main-theorem}---``only if part''} \label{section:onlyif}

  In this section, we show that the remaining higher rank Lie groups defining hermitian symmetric spaces
  \begin{gather*}
    \operatorname{SU}(n,m) \ (n, m \ge 2), \ \operatorname{SO}^0(5,2), \ \operatorname{SO}^*(2n) \ (n \ge 5), \ \operatorname{Sp}(n, \R) \ (n \ge 2)
  \end{gather*}
  
  do not exhibit non-profinite second bounded cohomology.  Recall that the group \(\operatorname{SO}^0(3,2)\) is isogenous to \(\operatorname{Sp}(2, \R)\) and the group \(\operatorname{SO}^0(4,2)\) is isogenous to \(\operatorname{SU}(2,2)\).

  \medskip
  As preparation, let \(k\) and \(l\) be totally real number fields and let \(\mathbf{G}\) and \(\mathbf{H}\) be simply-connected absolutely almost simple groups defined over \(k\) and \(l\), respectively.  Assume that \(\mathbf{G}\) is anisotropic at all infinite places of \(k\) except one which we call \(v\) and that \(\mathbf{H}\) is anisotropic at all infinite places of \(l\) except one which we call \(w\).  Suppose moreover that \(\operatorname{rank}_{k_v} \mathbf{G} \ge 2\) and \(\operatorname{rank}_{l_w} \mathbf{H} \ge 2\) and that there exist arithmetic subgroups \(\Gamma \le \mathbf{G}(k)\) and \(\Lambda \le \mathbf{H}(l)\) such that \(\widehat{\Gamma} \cong \widehat{\Lambda}\).  By adelic superrigidity~\cite{Kammeyer-Kionke:adelic}*{Theorem~3.4}, we have an isomorphism \(j \colon \mathbb{A}^f_l \rightarrow \mathbb{A}^f_k\) of topological rings and a group scheme isomorphism \(\eta \colon \mathbf{G} \times_k \mathbb{A}^f_k \longrightarrow \mathbf{H} \times_l \mathbb{A}^f_l\) over \(j\).  By the proof of~\cite{Klingen:arithmetical-similarities}*{Proposition~2.5\,(a), p.\,238}, the isomorphism \(j\) induces a bijection \(u \mapsto u'\) of the finite places of \(k\) and \(l\) and isomorphisms \(k_u \cong l_{u'}\).  Correspondingly, the isomorphism \(\eta\) over \(j\) splits into a family of isomorphisms \(\mathbf{G}(k_u) \cong \mathbf{H}(l_{u'})\).  In particular, \(\mathbf{G}\) and \(\mathbf{H}\) have the same Cartan Killing type.

  \begin{proposition} \label{prop:inner-twists}
    The \(\R\)-groups \(\mathbf{G}_v\) and \(\mathbf{H}_w\) are inner twists of one another.
  \end{proposition}

  \begin{proof}
    The case \(k = l = \Q\) was handled in \cite{Kammeyer-et-al:profinite-invariants}*{Proposition~2.7} (and it is apparent that the same arguments give the case of \(k=l\) if \(v=w\)).  For the general case, we argue as follows.  Let \(\mathbf{G_0}\) be the up to \(\Q\)-isomorphism unique \(\Q\)-split simply-connected \(\Q\)-group of the same Cartan Killing type as \(\mathbf{G}\) and \(\mathbf{H}\).  Let \(\mathbf{T} \subset \mathbf{G_0}\) be a maximal \(\Q\)-split torus, pick a set \(\Delta\) of simple roots of \(\mathbf{G_0}\) with respect to \(\mathbf{T}\), and let \(\operatorname{Sym} \Delta\) be the subgroup of the permutation group of \(\Delta\) given by Dynkin diagram symmetries.  Then we have a split short exact sequence
    \begin{equation} \label{eq:aut-sym} 1 \longrightarrow \operatorname{Ad} \mathbf{G_0} \xrightarrow{\ \iota \ } \operatorname{Aut} \mathbf{G_0} \xrightarrow{\ \pi \ } \operatorname{Sym} \Delta \longrightarrow 1 \end{equation}
    of \(\operatorname{Gal}(\Q)\)-groups where \(\operatorname{Gal}(\Q)\) acts trivially on \(\operatorname{Sym} \Delta\).  The group \(\mathbf{G}\) corresponds to a unique class \(\alpha \in H^1(k, \operatorname{Aut} \mathbf{G_0})\) and similarly \(\mathbf{H}\) corresponds to a unique class \(\beta \in H^1(l, \operatorname{Aut} \mathbf{G_0})\).  Since we have isomorphisms \(\mathbf{G}(k_u) \cong \mathbf{H}(l_{u'})\), it follows by functoriality that \(\pi_* \alpha \in H^1(k, \operatorname{Sym} \Delta)\) and \(\pi_* \beta \in H^1(l, \operatorname{Sym} \Delta)\) map diagonally to corresponding elements under the induced isomorphism
\[ \prod_{u \nmid \infty} H^1(k_u, \operatorname{Sym} \Delta) \cong \prod_{u' \nmid \infty} H^1(l_{u'}, \operatorname{Sym} \Delta). \]
Let us now first suppose that \(\mathbf{G}\) and hence \(\mathbf{G_0}\) does not have type \(D_4\).  We may then assume that \(\operatorname{Sym} \Delta \cong \Z / 2\Z\) because the proposition is trivial if \(\operatorname{Sym} \Delta\) is.  Note that the first Galois cohomology with coefficients in \(\mathbb{Z}/2\mathbb{Z}\) classifies quadratic field extensions.  So we conclude that \(\pi_* \alpha\) and \(\pi_* \beta\) correspond to quadratic extensions \(K/k\) and \(L/l\), respectively, such that \(\mathbb{A}^f_K \cong_{\mathbb{A}^f_\Q} \mathbb{A}^f_L\) (using \cite{Klingen:arithmetical-similarities}*{Theorem~2.3, p.\,237}).  This shows in particular that \(K\) and \(L\) have the same number of real embeddings~\cite{Klingen:arithmetical-similarities}*{Theorem~1.4\,(h), p.\,79}.  It follows that the number of real places in \(k\) extending to complex places in \(K\) equals the number of real places in \(l\) extending to complex places in \(L\).  Translating back from field extensions to Galois cohomology classes, this shows that \(\alpha\) and \(\beta\) localize to outer forms at the same number of infinite places.  Since \(\alpha\) and \(\beta\) localize to the compact real form (which may be inner or outer depending on the Cartan Killing type) at all other infinite places, \(\alpha_v\) and \(\beta_w\) must be either both outer or both inner forms.  In any case, they are inner twists of each other.

Now if \(\mathbf{G}\) and hence \(\mathbf{G_0}\) does have type \(D_4\), then \(\operatorname{Sym} \Delta \cong S_3 \cong \Z / 3 \rtimes \Z / 2 \).  Note that subgroups of the same order are conjugate in this group.  Therefore, the first Galois cohomology with coefficients in the trivial Galois module \(S_3\) classifies Galois extensions with Galois groups either \(\Z / 2\), or \(\Z / 3\), or \(S_3\) (or trivial).  So \(\pi_* \alpha\) and \(\pi_* \beta\) correspond to Galois extensions \(K/k\) and \(L/l\) of one and the same of these types, again such that \(\mathbb{A}^f_K \cong_{\mathbb{A}^f_\Q} \mathbb{A}^f_L\).  So once more, \(K\) and \(L\) have the same number of real embeddings.  As \(K\) and \(L\) are Galois over \(k\) and \(l\), respectively, real places extend either only to real places or only to complex places in these extensions.  Correspondingly, \(\alpha\) and \(\beta\) localize to outer forms again at the same number of infinite places of \(k\) and \(l\).  As all nontrivial homomorphisms \(\operatorname{Gal} (\R) \rightarrow \operatorname{Sym} \Delta\) are conjugate, any two real outer forms are inner twists of each other even in type \(D_4\).  It follows again that \(\alpha_v\) and \(\beta_w\) must be inner twists of one another.
\end{proof}

  \medskip
  Now let \(G\) be one of the groups listed at the beginning of this section and let \(\Gamma \le G\) be a lattice.  Let \(\Lambda \le H\) be a lattice in another higher rank Lie group and assume that \(\widehat{\Gamma} \cong \widehat{\Lambda}\).  We have to show that \(H\) also defines a hermitian symmetric space.  By Margulis arithmeticity~\cite{Margulis:discrete-subgroups}*{Chapter~IX, Theorem~1.11, p.\,298, (\(\ast \ast\)) and Remark~1.6.(i), pp.\,293-294}, there exists a \(k\)-group \(\mathbf{G}\) and an \(l\)-group \(\mathbf{H}\) with \(k\), \(l\), \(\mathbf{G}\), and \(\mathbf{H}\) as above such that \(\mathbf{G}_v\) is isogenous to \(G\), such that \(\mathbf{H}_w\) is isogenous to \(H\), and such that \(\mathbf{G}(\mathcal{O}_k)\) is commensurable with \(\Gamma\) while \(\mathbf{H}(\mathcal{O}_l)\) is commensurable with \(\Lambda\).  Since \(\Gamma\) and \(\Lambda\) are profinitely isomorphic, so are suitable finite index subgroups of \(\mathbf{G}(\mathcal{O}_k)\) and \(\mathbf{H}(\mathcal{O}_l)\).  Therefore, we have the conclusion from above that there exists an isomorphism \(j \colon \mathbb{A}^f_l \rightarrow \mathbb{A}^f_k\) and an isomorphism \(\eta \colon \mathbf{G} \times_k \mathbb{A}^f_k \longrightarrow \mathbf{H} \times_l \mathbb{A}^f_l\) over \(j\) and Proposition~\ref{prop:inner-twists} holds true for \(\mathbf{G}\) and~\(\mathbf{H}\).

\begin{proposition}
  The groups \(G \approx \operatorname{SU}(n,m)\) with \(n, m \ge 2\) do not exhibit non-profinite second bounded cohomology.
\end{proposition}

\begin{proof}
  By Proposition~\ref{prop:inner-twists}, the group \(\mathbf{H}_w\) is an inner twist of \(\mathbf{G}_v\), hence \(H\) is isogenous to \(\operatorname{SU}(n',m')\) for some \(n',m' \ge 2\) because the generalized special unitary groups are the only outer real forms of type \(A_n\) in the classification of real simple Lie groups.  So \(H\) defines a hermitian symmetric space, too.
\end{proof}

\begin{proposition}
  The groups \(G \approx \operatorname{Sp}(n, \R)\) with \(n \ge 2\) do not exhibit non-profinite second bounded cohomology.
\end{proposition}

\begin{proof}
  Let \(\mathbf{G_0} = \mathbf{Sp_n}\) be the unique \(\Q\)-split simply-connected \(\Q\)-group of type \(C_n\).  Then \(\operatorname{Sym} \Delta = 1\), so \(\operatorname{Aut} \mathbf{G_0} = \operatorname{Ad} \mathbf{G_0}\) and we have \(Z(\mathbf{G_0}) = \mu_2\).  From the short exact sequence of \(\operatorname{Gal}(\Q)\)-groups
  \[ 1 \longrightarrow Z(\mathbf{G_0}) \longrightarrow \mathbf{G_0} \longrightarrow \operatorname{Ad} \mathbf{G_0} \longrightarrow 1 \]
  we obtain a boundary map \(\delta_K \colon H^1(K, \operatorname{Ad} \mathbf{G_0}) \rightarrow H^2(K, \mu_2)\) for any field extension \(K/\Q\).  As we saw in~\cite{Kammeyer-Spitler:chevalley}*{Section~5}, the kernel of \(\delta_\R\) consists of the class corresponding to the split form \(\operatorname{Sp}(n, \R)\) only, while the other real forms of type \(C_n\) are the groups \(\operatorname{Sp}(p,q)\) with \(p + q = n\) and they form precisely the fiber under \(\delta_\R\) of the nontrivial element in \(H^2(\R, \mu_2) \cong \mathbb{Z} / 2 \mathbb{Z}\).  Let \(\alpha \in H^1(k, \operatorname{Ad} \mathbf{G_0})\) be the cohomology class corresponding to \(\mathbf{G}\) and let \(\beta \in H^1(l, \operatorname{Ad} \mathbf{G_0})\) be the cohomology class corresponding to \(\mathbf{H}\).  Then \(\delta_k(\alpha)\) and \(\delta_l(\beta)\) map diagonally to corresponding elements under the isomorphism
  \begin{equation} \label{eq:central} \prod_{u \neq v} H^2(k_u, \mu_2) \cong \prod_{u' \neq w} H^2(l_{u'}, \mu_2). \end{equation}
  Indeed, this follows from the isomorphism \(\eta\) for \(u \nmid \infty\).  Additionally, we know that at infinite places \(u \neq v\) and \(u' \neq w\), the groups \(\mathbf{G}_{k_u}\) and \(\mathbf{H}_{l_{u'}}\) are anisotropic hence isomorphic to the unique anisotropic real form \(\operatorname{Sp}(n)\) of type \(C_n\).  The number field version of the Albert--Brauer--Hasse--Noether theorem~\eqref{eq:brauer} says that there is a short exact sequence
  \begin{equation} \label{eq:brauer-general} 1 \longrightarrow \operatorname{Br}(k) \longrightarrow \bigoplus_u \operatorname{Br}(k_u) \longrightarrow \Q/\Z \longrightarrow 1. \end{equation}
  Since \(\mathbf{G}_v \approx \operatorname{Sp}(n,\R)\) is the real split form, \(\alpha_v\) is trivial.  Hence so is \(\delta_\R(\alpha_v)\) and therefore also \(\delta_\R(\beta_w)\) by~\eqref{eq:central} and~\eqref{eq:brauer-general}.  So \(\beta_w\) is trivial, too.  This shows that \(H \approx \operatorname{Sp}(n, \R)\) is hermitian.
\end{proof}

\begin{proposition}
  The groups \(G \approx \operatorname{SO}^*(2n)\) with \(n \ge 5\) do not exhibit non-profinite second bounded cohomology.
\end{proposition}

\begin{proof}
  In this case \(\mathbf{G}\) and \(\mathbf{H}\) have type \(D_n\) and it is more convenient to argue in terms of central simple algebras instead of Galois cohomology.  Indeed, it is then well-known by the work of Weil and Jacobson (see for instance \cite{Schoeneberg:semisimple}*{Theorem~4.5.10}) that there exist central simple algebras \(A\) and \(B\) over \(k\) and \(l\), respectively, each endowed with an involution \(\sigma\) of the first kind and of orthogonal type such that the skew symmetric Lie subalgebras of \(A\) and \(B\) consisting of the elements \(x\) satisfying \(\sigma(x) = -x\) with commutator Lie bracket \([x,y] = xy - yx\) are \(k\)- and \(l\)-isomorphic to the Lie algebras \(\mathfrak{g}\) of \(\mathbf{G}\) and \(\mathfrak{h}\) of \(\mathbf{H}\), respectively.  We have isomorphisms \(\mathfrak{g}_u \cong \mathfrak{h}_{u'}\) for all finite and infinite places \(u\) except possibly \(v\). From~\cite{Jacobson:lie-algebras}*{Chapter~X, Theorem~12} we conclude that these extend uniquely to isomorphisms \(A_u \cong B_{u'}\) which identify the involutions with one another.  In particular, \(A_u\) and \(B_{u'}\) have equal Brauer classes.  But since \(\mathfrak{g}_v \cong \mathfrak{so}^*(2n) = \mathfrak{so}(n, \mathbb{H})\), it follows that \(A_v\) has a non-trivial Brauer class, so by the Albert--Brauer--Hasse--Noether theorem~\eqref{eq:brauer-general}, the same is true for \(B_w\).  But all real forms of type \(D_n\) except \(\mathfrak{so}^*(2n)\) are special orthogonal lie algebras of quadratic forms over \(\mathbb{R}\), so the corresponding involutorial central simple algebra have trivial Brauer classes.  Hence of necessity \(\mathfrak{h}_w \cong \mathfrak{so}^*(2n)\) as well.  So \(H\) is isogenous to \(\operatorname{SO}^*(2n)\) and hence hermitian.
\end{proof}

This leaves the group \(G \approx \operatorname{SO}^0(5,2)\) as the only remaining case.  The only higher rank real twist of \(G\) up to isogeny is the group \(H \approx \operatorname{SO}^0(4,3)\).  However, this possibility can be excluded because the symmetric space defined by \(G\) is \(10\)-dimensional whereas the symmetric space defined by \(H\) is \(12\)-dimensional and the dimension mod four is a profinite invariant by~\cite{Kammeyer-et-al:profinite-invariants}*{Theorem~2.1}.

\end{document}